\documentclass[10pt]{amsart}
\usepackage{enumerate,amsmath,amssymb,latexsym,
amsfonts, amsthm, amscd, MnSymbol}

\theoremstyle{plain}

\newtheorem{theorem}{Theorem}

\newtheorem*{theorem*}{Theorem}

\numberwithin{equation}{section}

\newcommand{\R}{\mathbb{R}}

\newcommand{\sn}{{\rm sn}}

\newcommand{\ii}{{\rm i}}
\newcommand{\dd}{{\rm dn}_2}
\newcommand{\ddd}{{\rm dn}_3}

\newcommand{\pk}{p_{\kappa}}
\newcommand{\pl}{p_{\lambda}}
\newcommand{\ok}{\omega_{\kappa}}
\newcommand{\ol}{\omega_{\lambda}}

\begin{document}

\title {The B-B-G Transfer Principle for signature four}

\date{}

\author[P.L. Robinson]{P.L. Robinson}

\address{Department of Mathematics \\ University of Florida \\ Gainesville FL 32611  USA }

\email[]{paulr@ufl.edu}

 \keywords{}

\begin{abstract}
We show how the elliptic function $\dd$ of Shen leads to the signature four transfer principle of Berndt, Bhargava and Garvan. 
\end{abstract}

\maketitle

\medbreak 

Berndt, Bhargava and Garvan in [1995] established a Transfer Principle by which to pass from the classical theory of elliptic functions to the Ramanujan theory of elliptic functions in signature four. The essence of their transfer principle is contained in a pair of identities that relate the `classical' hypergeometric function $F_2 = F(\tfrac{1}{2}, \tfrac{1}{2}; 1 ; \bullet)$ to the `signature four' hypergeometric function $F_4 = F(\tfrac{1}{4}, \tfrac{3}{4}; 1 ; \bullet)$: explicitly, if $0 < x < 1$ then 
$$\sqrt{1 + x} \, F_4 (x^2) = F_2 ( 2x / (1 + x) )$$
and 
$$\sqrt{1 + x} \, F_4 (1 - x^2) = \sqrt2 \, F_2 ( (1 - x) / (1 + x) ).$$
These identities are derived in [1995] from hypergeometric identities recorded in the second notebook of Ramanujan: the first is an identity of Kummer, while the second also involves an identity due to Gauss. 

\medbreak 

Shen in [2014] revealed an actual elliptic function that naturally resides in signature four: his function $\dd$ involves incomplete integrals of the hypergeometric function $F(\tfrac{1}{4}, \tfrac{3}{4}; \tfrac{1}{2} ; \bullet)$; its construction is motivated by the way in which the classical Jacobian elliptic functions may be developed from $F(\tfrac{1}{2}, \tfrac{1}{2}; \tfrac{1}{2} ; \bullet)$. Among the results in [2014] are identifications of the fundamental periods of $\dd$ in terms of the hypergeometric function $F_4$; these identifications are established by explicit integral calculations. Also included in [2014] are formulae for $\dd$ and its companion functions ${\rm cn}_2$ (which is elliptic) and ${\rm sn}_2$ (which is not) in terms of the classical Jacobian functions ${\rm dn}$, ${\rm cn}$ and ${\rm sn}$ to a related modulus; these formulae are established with the aid of theta functions. 

\medbreak 

In this paper, we reconsider the elliptic function $\dd$ of Shen and thereby forge a new route to the signature four transfer principle of Berndt, Bhargava and Garvan. Our reconsideration of $\dd$ takes place entirely within the realm of elliptic functions, without the use of theta functions as intermediaries. The title of our paper notwithstanding, our primary aim is not to offer a new proof of the transfer principle: after all, the original proof by manipulation of hypergeometric functions is arguably more direct than a proof based on elliptic functions; rather, it is to reaffirm the status of $\dd$ as a natural elliptic function within the signature four theory. 

\medbreak 

The organization of this paper is as follows. In Section 1 we introduce the Shen elliptic function $\dd$ of modulus $\kappa$ and identify its coperiodic Weierstrass function $\pk$ in terms of its invariants. In Section 2 we identify the fundamental periods of $\dd$ and $\pk$ in terms of the signature four hypergeometric function $F_4$. In Section 3 we identify these fundamental periods in terms of the classical hypergeometric function $F_2$. Finally, in Section 4 we compare these two identifications of the periods, deducing the pair of identities that we displayed in our opening paragraph; these identities then yield the relationship between the signature four base $q_4$ and the classical base $q$ on which rests the Berndt-Bhargava-Garvan transfer principle. 

\bigbreak 

\section{The elliptic function $\dd$}

\medbreak 

Fix $\kappa \in (0, 1)$ as modulus, with corresponding (acute) modular angle $\alpha \in (0, \tfrac{1}{2} \pi)$ defined by $\sin \alpha = \kappa$ and with complementary modulus $\lambda \in (0, 1)$ defined by $\lambda = (1 - \kappa^2)^{1/2}$. The rule 
$$f(T) = \int_0^T F(\tfrac{1}{4}, \tfrac{3}{4}; \tfrac{1}{2} ; \kappa^2 \sin^2 t) \, {\rm d}t$$ 
defines a strictly increasing bijection $f : \R \to \R$. We write $\phi : \R \to \R$ for its inverse: thus, if $u \in \R$ then 
$$u = \int_0^{\phi (u)} F(\tfrac{1}{4}, \tfrac{3}{4}; \tfrac{1}{2} ; \kappa^2 \sin^2 t) \, {\rm d}t.$$ 
A subsidiary angular function with range $[- \alpha, \alpha]$ is then defined as the composite 
$$\psi = \arcsin (\kappa \sin \phi).$$

\medbreak 

Now, the function 
$$d_{\kappa} = \cos \psi : \R \to \R$$ 
has range $[\cos \alpha, 1] = [\lambda, 1]$ and satisfies the following initial value problem. 

\medbreak 

\begin{theorem} \label{d}
The function $d_{\kappa}$ has initial value $d_{\kappa} (0) = 1$ and satisfies the differential equation 
$$(d_{\kappa}')^2 = 2 (1 - d_{\kappa}) (d_{\kappa}^2 - \lambda^2).$$ 
\end{theorem} 

\begin{proof} 
The initial value is clear: $\phi(0) = 0$ so that $\psi(0) = 0$ and therefore $d(0) = 1$; here and below, we drop the subscript $\kappa$ when convenient. From $d = \cos \psi$ follows $d ' = - (\sin \psi) \psi'$; from $\sin \psi = \kappa \sin \phi$ follows $(\cos \psi) \psi ' = \kappa (\cos \phi) \phi '$; and from $f \circ \phi = {\rm id}$ follows 
$$\phi ' = \frac{1}{f ' \circ \phi} = \frac{1}{F(\tfrac{1}{4}, \tfrac{3}{4}; \tfrac{1}{2} ; \kappa^2 \sin^2 \phi)} = \frac{1}{F(\tfrac{1}{4}, \tfrac{3}{4}; \tfrac{1}{2} ; \sin^2 \psi)} = \frac{\cos \psi}{\cos \frac{1}{2} \psi}$$
on account of the standard hypergeometric identity 
$$F(\tfrac{1}{4}, \tfrac{3}{4}; \tfrac{1}{2} ; \sin^2 \psi) = \frac{\cos \frac{1}{2} \psi}{\cos \psi}$$
for which we refer to item (11) on page 101 in Volume 1 of the compendious Bateman Manuscript Project [1953]. Thus 
$$d ' = - \sin \psi \, \Big(\kappa \frac{\cos \phi}{\cos \psi}\Big) \, \frac{\cos \psi}{\cos \frac{1}{2} \psi} = - 2 \sin \tfrac{1}{2} \psi \, (\kappa \, \cos \phi)$$
and so 
$$(d ')^2 = 4 \sin^2 \tfrac{1}{2} \psi \, (\kappa^2 - \sin^2 \psi) = 2 (1 - \cos \psi) \, (\kappa^2 - 1 + \cos^2 \psi).$$
\end{proof} 

\medbreak 

The solution to this initial value problem is readily identifiable in Weierstrassian terms. 

\medbreak 

\begin{theorem} \label{p} 
The function $d_{\kappa} : \R \to \R$ satisfies 
$$(1 - d_{\kappa}) (\tfrac{1}{3} + \pk) = \tfrac{1}{2} \kappa^2$$ 
where $\pk = \wp( \bullet ; g_2, g_3)$ is the Weierstrass function with invariants 
$$g_2 = \lambda^2 + \tfrac{1}{3} \; \; {\and} \; \; g_3 = \tfrac{1}{3} \lambda^2 - \tfrac{1}{27}.$$ 
\end{theorem} 

\begin{proof} 
Either verify that the function 
$$p = - \tfrac{1}{3} + \frac{\frac{1}{2} \kappa^2}{1 - d}$$ 
has a pole at $0$ and satisfies the differential equation 
$$(p ')^2 = 4 p^3 - (\lambda^2 + \tfrac{1}{3}) p - (\tfrac{1}{3} \lambda^2 - \tfrac{1}{27})$$
or apply the argument that is to be found on page 453 in the classic treatise [1927] of Whittaker and Watson. 
\end{proof} 

\medbreak 

Thus, $d_{\kappa}$ is the restriction to $\R$ of an elliptic function: this is the elliptic function $\dd$ of Shen, given by 
$$\dd = 1 - \frac{\frac{1}{2} \kappa^2}{\tfrac{1}{3} + \pk} \, .$$

\medbreak 

The elliptic function $\dd$ and the Weierstrass function $\pk$ are evidently coperiodic. We shall write $(2 \ok, 2 \ok ')$ for their shared fundamental pair of periods such that $\ok > 0$ and $- \ii \, \ok ' > 0$. In the next two sections, we shall develop hypergeometric expressions for these periods. To close the present section, it is convenient to record the midpoint values of the Weiersstrass function $\pk$: in decreasing order, these zeros of the cubic 
$$4 e^3 - (\lambda^2 + \tfrac{1}{3}) e - (\tfrac{1}{3} \lambda^2 - \tfrac{1}{27})$$ 
are readily checked to be 
$$e_1 = \pk(\ok) = \tfrac{1}{6} + \tfrac{1}{2} \lambda$$
$$e_2 = \pk(\ok + \ok ') = \tfrac{1}{6} - \tfrac{1}{2} \lambda$$
$$e_3 = \pk(\ok ') = - \tfrac{1}{3}.$$

\medbreak 

\section{Fundamental periods in terms of $F_4$}

\medbreak 

The very definition of $\dd$ as an extension of $d_{\kappa}$ provides immediate access to the real half-period $\ok$ of $\dd$ and $\pk$. 

\medbreak 

\begin{theorem} \label{okF4}
$\ok = \tfrac{1}{2} \pi \, F(\tfrac{1}{4}, \tfrac{3}{4}; 1 ; \kappa^2).$
\end{theorem} 

\begin{proof} 
With 
$$I = \int_0^{\frac{1}{2} \pi} F(\tfrac{1}{4}, \tfrac{3}{4}; \tfrac{1}{2} ; \kappa^2 \sin^2 t) \, {\rm d} t$$
it may be verified by integration that 
$$\phi (u + 2 I ) = \phi (u) + \pi $$
so that 
$$\psi(u + 2 I) = - \psi(u)$$
and  
$$d (u + 2 I) = \cos \psi(u + 2 I) = \cos \psi (u) = d(u).$$ 
This shows that $\dd$ has $2 I$ as a period, which is easily seen to be least positive. Finally, expansion of the hypergeometric integrand and termwise integration show that 
$$\int_0^{\frac{1}{2} \pi} F(\tfrac{1}{4}, \tfrac{3}{4}; \tfrac{1}{2} ; \kappa^2 \sin^2 t) \, {\rm d} t= \tfrac{1}{2} \pi \, F(\tfrac{1}{4}, \tfrac{3}{4}; 1 ; \kappa^2).$$
\end{proof} 

\medbreak 

Access to the imaginary half-period $\ok '$ of $\dd$ and $\pk$ is facilitated by investigating the relationship between the primary Weierstrass function 
$$\pk = \wp (\bullet ; \ok, \ok ') = \wp (\bullet ; g_2, g_3)$$
and the auxiliary Weierstrass function 
$$q_{\kappa} = \wp (\bullet ; \ok, \tfrac{1}{2} \ok ') = \wp (\bullet ; h_2, h_3)$$
that results when its imaginary period is halved. Here, the invariants $h_2$ and $h_3$ of $q_{\kappa}$ are related to the invariants $g_2$ and $g_3$ of $\pk$ by 
$$h_2 = - 4 \, g_2 + 60 \, \pk (\ok ')^2$$
$$h_3 = 8 \, g_3 + 56 \, \pk (\ok ')^3.$$ 
This is a quite general consequence of the halving of a Weierstrassian period, for the proof of which we refer to Section 9.8 of [1989]. 

\medbreak 

\begin{theorem} \label{ok'F4}
$\ok ' = \ii \, \sqrt2 \, \tfrac{1}{2} \pi \, F(\tfrac{1}{4}, \tfrac{3}{4}; 1 ; 1 - \kappa^2).$
\end{theorem} 

\begin{proof} 
When the invariants of $\pk$ as displayed in Theorem \ref{p} and the subsequent evaluation $\pk(\ok ') = - 1/3$ are taken into account, we find that $q_{\kappa}$ has invariants 
$$h_2 = \tfrac{4}{3} + 4 \kappa^2 = (\ii \, \sqrt2)^4 (\kappa^2 + \tfrac{1}{3})$$ 
$$h_3 = \tfrac{8}{27} - \tfrac{8}{3} \kappa^2 = (\ii \, \sqrt2)^6 (\tfrac{1}{3} \kappa^2 - \tfrac{1}{27}).$$ 
By a further consultation of Theorem \ref{p} (but for the complementary modulus) in conjunction with the homogeneity relation for $\wp$ functions, we deduce that $q_{\kappa}$ is related to the Weierstrass function $p_{\lambda}$ of complementary modulus according to the rule 
$$q_{\kappa} (z) = - 2 \, p_{\lambda} (\ii \, \sqrt2 \, z).$$ 
Now on the one hand $q_{\kappa}$ has fundamental half-periods $\ok$ and $\tfrac{1}{2} \ok '$, while on the other hand $p_{\lambda}$ has fundamental half-periods $\ol$ and $\ol '$. In light of the above rule by which $q_{\kappa}$ and $\pl$ are related, we see that 
$$\ok ' = \ii \, \sqrt2 \, \ol.$$
It only remains to invoke Theorem \ref{okF4} (for the complementary modulus) and recall that $\lambda^2 = 1 - \kappa^2$. 

\end{proof} 

\medbreak 

\section{Fundamental periods in terms of $F_2$}

\medbreak 

In order to obtain equivalent expressions for $\ok$ and $\ok '$ in terms of the hypergeometric function $F(\tfrac{1}{2}, \tfrac{1}{2}; 1 ; \bullet)$ we shall reformulate the Weierstrass function $\pk$ in terms of classical Jacobian elliptic functions. 

\medbreak 

Recall from page 505 of [1927] that if the Weierstrass function $p$ has real midpoint values $e_1 > e_2 > e_3$ then 
$$p(z) = e_3 + \frac{e_1 - e_3}{\sn^2 [z (e_1 - e_3)^{1/2}]}$$
where $\sn = \sn (\bullet, k)$ is the Jacobian sine function with modulus $k \in (0, 1)$ given by  
$$k^2 = \frac{e_2 - e_3}{e_1 - e_3}$$
and its square $\sn^2$ has fundamental periods $(2 K, 2 \ii K')$ given by
$$K = \tfrac{1}{2}\pi \, F(\tfrac{1}{2}, \tfrac{1}{2}; 1 ; k^2) \; \; {\rm and} \; \; K' = \tfrac{1}{2}\pi \, F(\tfrac{1}{2}, \tfrac{1}{2}; 1 ; 1 - k^2).$$
Accordingly, $p$ itself has fundamental half-periods 
$$ \frac{K}{ (e_1 - e_3)^{1/2}} \; \; {\rm and} \; \; \ii \, \frac{K'}{ (e_1 - e_3)^{1/2}}.$$ 

\medbreak 

\begin{theorem} \label{okF2}
The half-periods $\ok$ and $\ok '$ of $\dd$ and $\pk$ are given by 
$$\sqrt{\tfrac{1 + \lambda}{2}} \; \ok = \tfrac{1}{2} \pi \, F(\tfrac{1}{2}, \tfrac{1}{2}; 1 ; \tfrac{1 - \lambda}{1 + \lambda})$$
and 
$$\sqrt{\tfrac{1 + \lambda}{2}} \; \ok ' = \ii \tfrac{1}{2} \pi \, F(\tfrac{1}{2}, \tfrac{1}{2}; 1 ; \tfrac{2 \lambda}{1 + \lambda}).$$
\end{theorem} 

\begin{proof} 
Apply to $p = \pk$ the foregoing recollections. As noted after Theorem \ref{p}, $\pk$ has midpoint values
$$e_1 = \tfrac{1}{6} + \tfrac{1}{2} \lambda, \, \, e_2 = \tfrac{1}{6} - \tfrac{1}{2} \lambda, \, \, e_3 = - \tfrac{1}{3}$$
so that 
$$k^2 = \frac{e_2 - e_3}{e_1 - e_3} = \frac{1 - \lambda}{1 + \lambda}$$
and 
$$1 - k^2 = \frac{2 \lambda}{1 + \lambda}$$
while 
$$(e_1 - e_3)^{1/2} = \sqrt{\tfrac{1 + \lambda}{2}}.$$ 
\end{proof} 

\medbreak

Of course, we may also use the assembled information to express the elliptic function $\dd$ in terms of the classical Jacobian elliptic functions to modulus $k$. Thus, the relation 
$$\pk (z) = - \tfrac{1}{3} + \frac{\tfrac{1}{2} (1 + \lambda)}{\sn^2 \Big[z \, (\tfrac{1}{2} (1 + \lambda))^{1/2}\Big]}$$
may be recast as 
$$\dd (z) = 1 - (1 - \lambda) \, \sn^2 \Big[(\tfrac{1}{2} (1 + \lambda))^{1/2} \Big];$$
equivalently, it may be recast either in terms of the Jacobian cosine function ${\rm cn}$ as 
$$\dd (z) = \lambda + (1 - \lambda) \, {\rm cn}^2 \Big[(\tfrac{1}{2} (1 + \lambda))^{1/2} \Big]$$
or in terms of the Jacobian `delta amplitude' ${\rm dn}$ as 
$$\dd (z) = - \lambda + (1 + \lambda)  \, {\rm dn}^2 \Big[(\tfrac{1}{2} (1 + \lambda))^{1/2} \Big].$$
Incidentally, it may be checked that the Jacobian modulus $k$ equals $\tan \tfrac{1}{2} \alpha$. 

\medbreak 

\section{The transfer principle}

\medbreak 

All the pieces are in place: we are now in a position to deduce the hypergeometric identities that opened our paper. 

\medbreak 

\begin{theorem} \label{hyp}
If $0 < \lambda < 1$ then 
$$\sqrt{1 + \lambda} \,  F(\tfrac{1}{4}, \tfrac{3}{4}; 1 ; 1 - \lambda^2) =  \sqrt2 \, F(\tfrac{1}{2}, \tfrac{1}{2}; 1 ; \frac{1 - \lambda}{1 + \lambda})$$ 
and 
$$\sqrt{1 + \lambda} \,  F(\tfrac{1}{4}, \tfrac{3}{4}; 1 ; \lambda^2) =  F(\tfrac{1}{2}, \tfrac{1}{2}; 1 ; \frac{2 \lambda}{1 + \lambda})\, .$$
\end{theorem} 

\begin{proof} 
Direct comparison of Theorem \ref{okF4} with the first formula of Theorem \ref{okF2} yields 
$$\sqrt{1 + \lambda} \,  F(\tfrac{1}{4}, \tfrac{3}{4}; 1 ; \kappa^2) =  \sqrt2 \, F(\tfrac{1}{2}, \tfrac{1}{2}; 1 ; \frac{1 - \lambda}{1 + \lambda})$$ 
while direct comparison of Theorem \ref{ok'F4} with the second formula of Theorem \ref{okF2} yields 
$$\sqrt{1 + \lambda} \,  F(\tfrac{1}{4}, \tfrac{3}{4}; 1 ; 1 - \kappa^2) =  F(\tfrac{1}{2}, \tfrac{1}{2}; 1 ; \frac{2 \lambda}{1 + \lambda})\, .$$ 
\end{proof} 

\medbreak 

As in [1995] these hypergeometric identities entail a connexion between the base $q_4$ that is appropriate to the signature four elliptic theory and the base $q$ that is appropriate to the classical elliptic theory. To be explicit, Theorem \ref{hyp} implies that 
$$\frac{F(\tfrac{1}{4}, \tfrac{3}{4}; 1 ; 1 - \lambda^2)}{F(\tfrac{1}{4}, \tfrac{3}{4}; 1 ; \lambda^2)} = \sqrt2 \, \frac{F(\tfrac{1}{2}, \tfrac{1}{2}; 1 ; \frac{1 - \lambda}{1 + \lambda})}{F(\tfrac{1}{2}, \tfrac{1}{2}; 1 ; \frac{2 \lambda}{1 + \lambda})}$$
whence 
$$q_4 \, (\lambda^2) : = \exp \Big\{ - \pi \sqrt2 \, \frac{F(\tfrac{1}{4}, \tfrac{3}{4}; 1 ; 1 - \lambda^2)}{F(\tfrac{1}{4}, \tfrac{3}{4}; 1 ; \lambda^2)} \Big\}$$
and 
$$q \, \big(\tfrac{2 \lambda}{1 + \lambda}\big) : = \exp \Big\{ - \pi \, \frac{F(\tfrac{1}{2}, \tfrac{1}{2}; 1 ; 1 - \frac{2 \lambda}{1 + \lambda}))}{F(\tfrac{1}{2}, \tfrac{1}{2}; 1 ; \frac{2 \lambda}{1 + \lambda})}\Big\}$$
satisfy the relation 
$$q_4 \, (\lambda^2) = q \, \big(\tfrac{2 \lambda}{1 + \lambda}\big)^2 .$$
\medbreak 

The signature four transfer principle now follows exactly as in [1995]. 

\medbreak 

\begin{center} 
{\small R}{\footnotesize EFERENCES}
\end{center} 
\medbreak

[1927] E.T. Whittaker and G.N. Watson, {\it A Course of Modern Analysis}, Fourth Edition, Cambridge University Press. 

\medbreak 

[1953] A. Erdelyi (director), {\it Higher Transcendental Functions}, Volume 1, McGraw-Hill. 

\medbreak 

[1989] D.F. Lawden, {\it Elliptic Functions and Applications}, Applied Mathematical Sciences {\bf 80}, Springer-Verlag. 

\medbreak 

[1995] B.C. Berndt, S. Bhargava, and F.G. Garvan, {\it Ramanujan's theories of elliptic functions to alternative bases}, Transactions of the American Mathematical Society {\bf 347} 4163-4244. 

\medbreak 

[2014] Li-Chien Shen, {\it On a theory of elliptic functions based on the incomplete integral of the hypergeometric function $_2 F_1 (\frac{1}{4}, \frac{3}{4} ; \frac{1}{2} ; z)$}, Ramanujan Journal {\bf 34} 209-225.

\medbreak

\end{document}